\documentclass[reqno,11pt]{amsart}
\usepackage{amsmath}
\usepackage{amssymb}
\usepackage{latexsym}
\usepackage{graphicx}
\usepackage[hypertex]{hyperref}
\usepackage{amsmath,amssymb,amsthm,amsfonts,hyperref}

\hypersetup{colorlinks,citecolor=red,pdfstartview=FitH, pdfpagemode=None}
\newtheorem{theorem}{Theorem}[section]

\newtheorem{proposition}[theorem]{Proposition}
\theoremstyle{definition}
\newtheorem{definition}[theorem]{Definition}
\newtheorem{example}[theorem]{Example}

\newtheorem{corollary}[theorem]{Corollary}

\newtheorem{remark}[theorem]{Remark}
\numberwithin{equation}{section}

\input{tcilatex}

\begin{document}
\title[]{SOFT $\beta $-OPEN SETS AND THEIR APPLICATIONS}
\author{Y. YUMAK, A. K. KAYMAKCI}
\address{Department of Mathematics, Sel\c{c}uk University, 42075 Campus,
Konya, Turkey}
\email{yunusyumak@selcuk.edu.tr, akeskin@selcuk.edu.tr }
\maketitle

\begin{abstract}
First of all, we focused soft $\beta $-open sets and soft $\beta $-closed
sets over the soft topological space and investigated some properties of
them. Secondly, we defined the concepts soft $\beta $-continuity, soft $%
\beta $-irresolute and soft $\beta $-homeomorphism on soft topological
spaces. We also obtained some characterizations of these mappings. Finally,
we observed that the collection $S\beta r$-$h(X,\tau ,E)$ was a soft group.
\end{abstract}

\section{\texttt{\ }INTRODUCTION}

Molodtsov $\left[ 1\right] $ initiated the theory of soft sets as a new
mathematical tool for dealing with uncertainties which traditional
mathematical tools cannot hold. He has shown several aplications of this
theory in solving many practical problems in economics, engineering, social
science, medical science, etc.

In the recent years, papers about soft sets theory and their applications in
various fields have been writing increasingly. Shabir and Naz $\left[ 3%
\right] $ introduced the notions of soft topological spaces which are
defined over an initial universe with a fixed set of parameters. The authors
introduced the definition of soft open sets, soft closed sets, soft interior
, soft closure, and soft seperation axioms. Chen $\left[ 6\right] $
introduced soft semi open sets and related properties. Gunduz Aras et al. $%
\left[ 9\right] $ introduced soft continuous mappings which are defined over
an initial universe set with a fixed set of parameters. Mahanta and Das $[8]$
introduced and characterized various forms of soft functions, like
semicontinuous, irresolute, semiopen soft functions. And Arockiarani and
Lancy $[7]$ introduced soft $g\beta $-closed sets and soft $gs\beta $-closed
sets in soft topological spaces and they obtained some properties in the
light of these defined sets.

In the present study, first of all, we have focused soft $\beta $-open and
soft $\beta $-closed sets over the soft topological space and investigated
some properties of them. Secondly, we have defined the concepts soft $\beta $%
-continuity, soft $\beta $-irresolute and soft $\beta $-homeomorphism on
soft topological spaces. We have also obtained some characterizations of
these mappings. Finally, we have observed that the collection $S\beta r$-$%
h(X,\tau ,E)$ was a soft group.

\section{PRELIMINARIES}

Let $U$ be an initial universe set and $E$ be a collection of all possible
parameters with respect to $U$, where parameters are the characteristics or
properties of objects in $U$. Let $P(U)$ denote the power set of $U$, and
let $\vspace{0in}\smallskip \allowbreak A{}\subseteq $ $E$.

\begin{definition}
$\left( \left[ 1\right] \right) $. A pair $(F,A)$ is called a soft set over $%
U$, where $F$ is a mapping given by 
\begin{equation*}
F:A\longrightarrow P(U)
\end{equation*}%
In other words, a soft set over $U$ is a parametrized family of subsets of
the universe $U$. For a particular $e\in A$. $F(e)$ may be considered the
set of $e$-approximate elements of the soft set $(F,A)$.
\end{definition}

\begin{definition}
$\left( \left[ 4\right] \right) $. For two soft sets $(F,A)$ and $(G,B)$
over a common universe $U$, we say that $(F,A)$ is a soft subset of $(G,B)$
if
\end{definition}

(i) $A\subseteq B$, and

(ii)$\forall e\in A$, $F(e)\subseteq G(e)$.

\ \ \ \ \ \ \ \ \ \ \ \ \ \ \ \ \ \ \ \ \ \ \ \ \ \ \ \ \ \ \ \ \ \ \ \ \ \
\ \ \ \ \ \ \ \ \ \ \ \ 

We write $(F,A)$ $\widetilde{\subseteq }$ $(G,B)$. $(F,A)$ is said that to
be a soft super set of $(G,B)$, if $(G,B)$ is a soft subset of $(F,A)$. We
denote it by $(F,A)$ $\widetilde{\supseteq }$ $(G,B)$.

\begin{definition}
$\left( \left[ 2\right] \right) $. A soft set $(F,A)$ over $U$ is said to be
\end{definition}

$(i)$ null soft set denoted by $\Phi $, if $\forall e\in A$, $F(e)=\phi $.

$(ii)$ absolute soft set denoted by $\widetilde{A}$, if $\forall e\in A$, $%
F(e)=U$.

\begin{definition}
For two soft sets $(F,A)$ and $(G,B)$ over a common universe $U$,
\end{definition}

$(i)$ $\left( \left[ 2\right] \right) $ union of two soft sets of $(F,A)$
and $(G,B)$ is the soft set $(H,C)$, where $C=A\cup B$, and $\forall e\in C$,

\begin{equation*}
H(e)=\left\{ 
\begin{array}{l}
F(e)\text{ \ \ \ \ \ \ } \\ 
G(e)\text{ \ \ \ \ \ } \\ 
F(e)\cup G(e)%
\end{array}%
\right. 
\begin{array}{l}
, \\ 
, \\ 
,%
\end{array}%
\begin{array}{l}
\text{\ if }e\in A-B\text{\ \ \ \ \ } \\ 
\text{\ if }e\in B-A \\ 
\text{\ if }e\in A\cap B%
\end{array}%
\end{equation*}

We write $(F,A)$ $\widetilde{\cup }$ $(G,B)=(H,C)$.

\ \ \ \ \ \ \ \ \ 

$(ii)$ $\left( \left[ 4\right] \right) $ intersection of $(F,A)$ and $(G,B)$
is the soft set $(H,C)$, where $C=A\cap B$, and $\forall e\in C$, $%
H(e)=F(e)\cap G(e)$. We write $(F,A)$ $\widetilde{\cap }$ $(G,B)=(H,C)$.

\ \ \ \ \ \ \ \ \ \ \ \ \ \ \ \ \ \ \ \ \ \ \ \ \ \ \ \ \ \ \ \ \ \ 

Let $X$ be an initial universe set and $E$ be the non-empty set of
parameters.

\begin{definition}
$\left( \left[ 3\right] \right) $. Let $(F,E)$ be a soft set over $X$ and $%
x\in X$ . We say that $x\in (F,E)$ read as $x$ belongs to the soft set $%
(F,E) $ whenever $x\in F(e)$ for all $e\in E$. Note that for any $x\in X$ . $%
x\notin (F,E),$ if $x\notin F(e)$ for some $e\in E$.
\end{definition}

\begin{definition}
$\left( \left[ 3\right] \right) $. Let $Y$ be a non-empty subset of $X$,
then $\widetilde{Y}$ denotes the soft set $(Y,E)$ over $X$ for which $%
Y(e)=Y, $ for all $e\in E$. In particular, $(X,E),$ will be denoted by $%
\widetilde{X}.$
\end{definition}

\begin{definition}
$\left( \left[ 3\right] \right) $. The relative complement of a soft set $%
(F,E)$ is denoted by $(F,E)^{\prime }$and is defined by $(F,E)^{\prime
}=(F^{\prime },E)$ where $F^{\prime }:E\longrightarrow P(U)$ is a mapping
given by $F^{\prime }(e)=U-F(e)$ for all $e\in E$.
\end{definition}

\begin{definition}
$\left( \left[ 3\right] \right) $. Let $\tau $ be the collection of soft
sets over $X$, then $\tau $ is said to be soft topology on $X$ if
\end{definition}

(1) $\Phi $, $\widetilde{X}$ are belong to $\tau $

(2) the union of any number of soft sets in $\tau $ belongs to $\tau $

(3) the intersection of any two soft sets in $\tau $ belongs to $\tau $

\bigskip The triplet $(X,\tau ,E)$ is called a soft topological space over $%
X.$ The members of $\tau $ are said to be soft open sets in $X$.

\begin{definition}
$\left( \left[ 3\right] \right) $. Let $(X,\tau ,E)$ be a soft topological
space over $X$. A soft set $(F,E)$ over $X$ is said to be a soft closed set
in $X$, if its relative complement $(F,E)^{\prime }$ belongs\ to $\tau .$
\end{definition}

\begin{definition}
Let $(X,\tau ,E)$ be a soft topological space over $X$ and $(F,E)$ be a soft
set over $X$. Then
\end{definition}

$(i)$ $\left( \left[ 5\right] \right) $ soft interior of soft set $(F,E)$ is
denoted by $(F,E)^{\circ }$ and is defined as the union of all soft open
sets contained in $(F,E)$. Thus $(F,E)^{\circ }$ is the largest soft open
set contained in $(F,E)$.

\bigskip $(ii)$ $\left( \left[ 3\right] \right) $ soft closure of $(F,E)$,
denoted by $\overline{(F,E)}$ is the intersection of all soft closed super
sets of $(F,E)$. Clearly $\overline{(F,E)}$ is the smallest soft closed set
over $X$ which contains $(F,E)$.

We will denote interior(resp. closure) of soft set $(F,E)$ as $int\left(
F,E\right) $(resp.$cl\left( F,E\right) $).

\begin{proposition}
$\left( \left[ 5\right] \right) $. Let $(X,\tau ,E)$ be a soft topological
space over $X$ and $(F,E)$ and $(G,E)$ be a soft set over $X$. Then
\end{proposition}

(1) $int\left( int\left( F,E\right) \right) =int\left( F,E\right) $

(2) $(F,E)$ $\widetilde{\subseteq }$ $(G,E)$ imples $int\left( F,E\right) $ $%
\widetilde{\subseteq }$ $int(G,E)$

(3) $cl\left( cl\left( F,E\right) \right) =cl\left( F,E\right) $

(4) $(F,E)$ $\widetilde{\subseteq }$ $(G,E)$ imples $cl\left( F,E\right) $ $%
\widetilde{\subseteq }$ $cl(G,E)$

\section{Soft $\protect\beta $-open sets and soft $\protect\beta $-closed
sets}

\begin{definition}
A soft set $(F,E)$ in a soft topological space $(X,\tau ,E)$ is said to be
\end{definition}

$(i)$ $\left( \left[ 6\right] \right) $ soft $semi$-open if\ $(F,E)$ $%
\widetilde{\subseteq }$ $cl(int(F,E))$.

$(ii)$ $\left( \left[ 7\right] \right) $ soft $pre$-open if \ $(F,E)$ $%
\widetilde{\subseteq }$ $int(cl(F,E))$.

$(iii)$ $\left( \left[ 7\right] \right) $ soft $\alpha $-open if $(F,E)$ $%
\widetilde{\subseteq }$ $int(cl(int(F,E)))$.

$(iv)$ $\left( \left[ 7\right] \right) $ soft $\beta $-open if $(F,E)$ $%
\widetilde{\subseteq }$ $cl(int(cl(F,E)))$.

$(v)$ $\left( \left[ 7\right] \right) $ soft $\beta $-closed if $\
int(cl(int(F,E)))$ $\widetilde{\subseteq }$ $(F,E)$.

The family of all soft $\beta $-open sets (resp. soft $\beta $-closed sets)
in a soft topological space $(X,\tau ,E)$ will be denote by $S.\beta .O(X)$
(resp. $S.\beta .C(X))$.

\begin{remark}
It is clear that $S.\beta .O(X)$ contains each of $S.S.O(X),$ $S.P.O(X)$ and 
$S.\alpha .O(X)$, the following diagram shows this fact.
\end{remark}

soft open set $\longrightarrow $\ \ soft $\alpha $-open set $\
\longrightarrow $ \ soft semi-open set\ \ 

$\ \ \ \ \ \ \ \ \ \ \ \ \ \ \ \ \ \ \ \ \ \ \ \ \ \ \ \ \ \ \ \ \ \ \
\downarrow $ \ \ \ \ \ \ \ \ \ \ \ \ \ \ \ \ \ \ \ \ \ \ \ \ \ \ \ $%
\downarrow $

\bigskip\ \ \ \ \ \ \ \ \ \ \ \ \ \ \ \ \ \ \ \ \ \ soft $pre$-open set\ \ \ 
$\longrightarrow $\ \ \ soft $\beta $-open set

\ \ \ \ \ \ \ \ \ \ \ \ \ \ \ \ \ \ \ \ \ \ \ \ \ \ \ \ \ \ \ \ \ \ \ \ \ \
\ \ \ \ \ \ \ \ \ \ \ \ \ \ \ \ \ \ \ \ \ \ \ \ \ \ \ \ \ \ \ \ \ \ \ \ \ \
\ \ \ \ \ \ \ \ \ \ \ \ \ \ \ \ \ \ \ \ \ \ \ \ \ \ \ \ \ \ \ \ \ \ \ \ \ \
\ \ \ \ \ \ \ \ \ \ \ \ \ \ 

The converses need not be true, in general, as show in the following example.

\begin{example}
Let $X=\{x_{1},x_{2},x_{3}\},$ $E=\{e_{1},e_{2}\}$ and $\tau =\{\Phi ,%
\widetilde{X},(F_{1},E),(F_{2},E),...,(F_{7},E)\}$ where $%
(F_{1},E),(F_{2},E),........,(F_{7},E)$ are soft sets over $X,$defined as
follows .
\end{example}

\ \ \ \ \ \ \ \ \ \ \ \ \ \ \ \ \ \ \ \ \ \ \ \ \ \ \ \ 

$%
\begin{array}{ll}
F_{1}(e_{1})=\{x_{1},x_{2}\} & \text{ \ \ \ }F_{1}(e_{2})=\{x_{1},x_{2}\} \\ 
F_{2}(e_{1})=\{x_{2}\} & \text{ \ \ \ }F_{2}(e_{2})=\{x_{1},x_{3}\} \\ 
F_{3}(e_{1})=\{x_{2},x_{3}\} & \text{ \ \ \ }F_{3}(e_{2})=\{x_{1}\} \\ 
F_{4}(e_{1})=\{x_{2}\} & \text{ \ \ \ }F_{4}(e_{2})=\{x_{1}\} \\ 
F_{5}(e_{1})=\{x_{1},x_{2}\} & \text{ \ \ \ }F_{5}(e_{2})=X \\ 
F_{6}(e_{1})=X\  & \text{ \ \ \ }F_{6}(e_{2})=\{x_{1},x_{2}\} \\ 
F_{7}(e_{1})=\{x_{2},x_{3}\} & \text{ \ \ \ }F_{7}(e_{2})=\{x_{1},x_{3}\}%
\end{array}%
$

\ \ \ \ \ \ \ \ \ \ \ \ \ \ \ \ \ \ \ \ \ \ \ \ \ \ \ \ \ \ \ \ \ 

Then $\tau $ defines a soft topology on $X$ and hence $(X,\tau ,E)$ is a
soft topological space over $X$ . Now we give a soft set $(H,E)$ in $(X,\tau
,E)$ is defined as follows:

\ \ \ \ \ \ \ \ \ \ \ \ \ \ \ \ \ \ 

$H(e_{1})=\phi ,$ $\ \ H(e_{2})=\{x_{1}\}$

\ \ \ \ \ \ \ \ \ \ \ \ \ \ \ \ \ \ \ \ \ \ \ \ \ \ \ \ \ 

Then, $(H,E)$ is a soft $pre$-open set \ but not a soft $\alpha $-open set,
also it is a soft $\beta $-open set but not soft $semi$-open set .

\bigskip\ \ \ \ \ \ \ \ \ \ \ \ \ \ \ \ \ \ \ \ 

Now we define the notion of soft supratopology is weaker than soft topology.

\begin{definition}
Let $\tau $ be the collection of soft sets over $X$, then $\tau $ is said to
be soft supratopology on $X$ if
\end{definition}

$(1)$ $\Phi $, $\widetilde{X}$ are belong to $\tau $

$(2)$ the union of any number of soft sets in $\tau $ belongs to $\tau $

\ \ \ \ \ \ \ 

We give the following property for soft $\beta $-open sets.

\begin{proposition}
\bigskip The collection $S.\beta .O(X)$ of all soft $\beta $-open sets of a
space $(X,\tau ,E)$ form a soft supratopology.
\end{proposition}

\begin{proof}
(1) is obvious

(2)Let $(F_{i},E)$ $\in $ $S.\beta .O(X)$ for $\forall i\in I=\{1,2,3.....\}$%
. Then, for $\forall i\in I$

\ \ \ \ \ \ \ \ \ \ \ \ \ \ \ \ \ \ \ \ \ \ \ \ \ \ \ \ \ \ \ \ \ \ \ \ \ \
\ \ \ \ \ \ \ \ \ \ \ \ \ \ \ \ \ \ \ \ \ \ \ \ \ \ \ \ \ \ \ \ \ \ \ \ \ \
\ \ \ \ \ \ \ \ \ \ \ \ \ \ \ \ \ \ \ \ \ \ \ \ \ \ \ \ \ \ \ \ \ \ \ 

$%
\begin{array}{llll}
(F_{i},E)\text{ }\widetilde{\subseteq }\text{ }cl(int(cl(F_{i},E))) & 
\Longrightarrow & \underset{i\in I}{\widetilde{\cup }}(F_{i},E) & \widetilde{%
\subseteq }\underset{i\in I}{\widetilde{\cup }}(cl(int(cl(F_{i},E)))) \\ 
&  &  & =cl(\underset{i\in I}{\widetilde{\cup }}(int(cl(F_{i},E)))) \\ 
&  &  & \widetilde{\subseteq }\text{ }cl(int(\underset{i\in I}{\widetilde{%
\cup }}(cl(F_{i},E)))) \\ 
&  &  & =cl(int(cl(\underset{i\in I}{\widetilde{\cup }}(F_{i},E))))%
\end{array}%
$
\end{proof}

The intersection of two soft $\beta $-open sets need not be soft $\beta $%
-open set as is illustrated by the following example.

\begin{example}
Let $X=\{x_{1},x_{2}\}$, $E=\{e_{1},e_{2}\}$ and $\tau =\{\Phi ,\widetilde{X}%
,(F_{1},E),(F_{2},E),(F_{3},E)\}$ where $(F_{1},E),(F_{2},E),(F_{3},E)$ are
soft sets over $X,$defined as follows.
\end{example}

$%
\begin{array}{ll}
F_{1}(e_{1})=\{x_{1}\} & F_{1}(e_{2})=\{x_{2}\} \\ 
F_{2}(e_{1})=\{x_{1},x_{2}\} & F_{2}(e_{2})=\{x_{2}\} \\ 
F_{3}(e_{1})=\{x_{1}\} & F_{3}(e_{2})=\{x_{1},x_{2}\}%
\end{array}%
$

\ \ \ \ \ \ \ \ \ \ \ \ \ \ \ \ \ \ \ \ \ \ \ \ \ \ \ \ \ \ \ \ \ \ \ \ \ \
\ \ \ \ \ \ \ \ \ \ \ \ \ \ \ \ \ \ \ \ \ \ \ \ \ \ \ \ \ \ \ \ \ \ \ \ \ \
\ \ \ \ \ \ \ \ \ \ \ \ \ \ \ \ \ \ \ \ \ \ \ \ \ \ \ \ \ \ \ \ \ \ \ \ \ \
\ \ \ \ \ \ \ \ \ \ \ \ \ \ \ \ \ \ \ \ \ \ \ \ \ \ \ \ \ \ \ \ \ \ \ \ \ \
\ \ \ \ \ \ \ \ \ \ \ \ \ \ \ \ \ \ \ \ \ \ \ \ \ \ \ \ \ \ \ \ \ \ \ \ \ \
\ \ \ \ \ \ \ \ \ \ \ \ \ \ \ \ \ \ \ \ \ \ \ \ \ \ \ \ \ \ \ \ \ \ \ \ \ \
\ \ \ \ \ \ \ \ \ \ \ \ \ \ \ \ \ \ \ \ \ 

Then $\tau $ defines a soft topology on $X$ and hence $(X,\tau ,E)$ is a
soft topological space over $X$ .Now we give two soft sets $(G,E),$ $(H,E)$
in $(X,\tau ,E)$ is defined as follows:

\ \ \ \ \ \ \ \ \ \ \ \ \ \ \ \ \ \ \ \ \ \ \ \ \ \ \ \ \ \ \ \ \ \ \ \ \ \
\ \ \ \ \ \ \ \ \ \ \ \ \ \ \ \ \ \ \ \ \ \ 

$%
\begin{array}{ll}
G(e_{1})=\{x_{2}\} & G(e_{2})=\{x_{2}\} \\ 
H(e_{1})=\{x_{1},x_{2}\} & H(e_{2})=\{x_{1}\}%
\end{array}%
$

\ \ \ \ \ \ \ \ \ \ \ \ \ \ \ \ \ \ \ \ \ \ \ \ \ \ \ \ \ \ \ \ \ \ \ \ \ \
\ \ \ \ \ \ \ \ \ \ \ \ \ \ \ \ \ \ \ \ \ \ \ \ \ \ \ \ \ \ \ \ \ \ \ \ \ \
\ \ \ \ \ \ \ \ 

Then, $(G,E)$ and $(H,E)$ are soft $\beta $-open sets over $X$, therefor,

$(G,E)$ $\widetilde{\cap }$ $(H,E)=\{\{x_{2}\},\phi \}=(K,E)$ and $%
cl(int(cl(K,E)))=$ $\Phi $ . Hence, $(K,E)$ is not a soft $\beta $-open set.

\bigskip\ \ \ \ \ \ \ \ \ \ \ \ \ \ 

We have the following proposition by using relative complement.

\begin{proposition}
Arbitrary intersection of soft $\beta $-closed sets is soft $\beta $-closed.
\end{proposition}

\begin{proof}
Let $(F_{i},E)\in S.\beta .C(X)$ for $\forall i\in I=\{1,2,3.....\}$ .Then,
for $\forall i\in I,$

\ \ \ \ \ \ \ \ \ \ \ \ \ \ \ \ \ \ \ \ \ \ \ \ \ \ \ \ \ \ \ \ \ \ \ \ \ \
\ \ \ \ \ \ \ \ \ \ \ \ \ \ \ \ \ \ \ \ \ \ \ \ \ \ \ \ \ \ \ \ \ \ \ \ \ \
\ \ \ \ \ \ \ \ \ \ \ \ \ \ \ \ \ \ \ \ \ \ \ \ \ \ \ \ \ \ \ \ \ \ \ \ \ \
\ \ \ \ \ \ \ \ \ \ \ \ \ \ \ \ \ \ \ \ \ \ \ \ \ \ \ \ \ \ \ \ \ \ \ \ \ \
\ \ \ \ \ \ \ \ \ \ \ \ \ \ \ \ \ \ \ \ \ \ \ \ \ \ \ \ \ \ \ \ \ \ \ \ \ \
\ \ \ \ 

$%
\begin{array}{llll}
(F_{i},E)\text{ }\widetilde{\supseteq }\text{ }int(cl(int(F_{i},E))) & 
\Longrightarrow & \underset{i\in I}{\widetilde{\cap }}(F_{i},E) & \widetilde{%
\supseteq }\text{ }\underset{i\in I}{\widetilde{\cap }}%
(int(cl(int(F_{i},E)))) \\ 
&  &  & =int(\underset{i\in I}{\widetilde{\cap }}(cl(int(F_{i},E)))) \\ 
&  &  & \widetilde{\supseteq }\text{ }int(cl(\underset{i\in I}{\widetilde{%
\cap }}(int(F_{i},E)))) \\ 
&  &  & =int(cl(int(\underset{i\in I}{\widetilde{\cap }}(F_{i},E))))%
\end{array}%
$
\end{proof}

The union of two soft $\beta $-closed sets need not be soft $\beta $-closed
set as is illustrated by the following example.

\begin{example}
Let $(X,\tau ,E)$ as in Example 3.6. Now we give two soft sets $(G,E),$ $%
(H,E)$ in $(X,\tau ,E)$ is defined as follows:
\end{example}

\ \ \ \ \ \ \ \ \ \ \ \ \ \ \ \ \ \ \ \ \ \ \ \ \ \ \ \ \ \ \ \ \ \ \ \ \ \
\ \ \ \ \ \ \ \ \ \ \ \ \ \ \ \ \ \ \ \ \ \ 

$%
\begin{array}{ll}
G(e_{1})=\{x_{1}\} & G(e_{2})=\{x_{1}\} \\ 
H(e_{1})=\phi & H(e_{2})=\{x_{2}\}%
\end{array}%
$

\ \ \ \ \ \ \ \ \ \ \ \ \ \ \ \ \ \ \ \ \ \ \ \ \ \ \ \ \ \ \ \ \ \ \ \ \ \
\ \ \ \ \ \ \ \ \ \ \ \ \ \ \ \ \ \ \ \ \ \ \ \ \ \ \ \ \ \ \ \ \ \ \ \ \ \
\ \ \ \ \ \ 

Then, $(G,E)$ and $(H,E)$ are soft $\beta $-closed sets over $X$ , therefor,

$(G,E)$ $\widetilde{\cup }$ $(H,E)=\{\{x_{1}\},\{x_{1},x_{2}\}\}=(K,E)$ and $%
int(cl(int(K,E)))=\widetilde{X}$ . Hence, $(K,E)$ is not a soft $\beta $%
-closed set.

\begin{theorem}
Each soft $\beta $-open set which is soft semi-closed is soft semi-open .
\end{theorem}

\begin{proof}
$(F,E)\in S.\beta .O(X)\Longrightarrow (F,E)$ $\widetilde{\subseteq }$ $%
cl(int(cl(F,E)))$ and $(F,E)\in S.S.C(X)\Longrightarrow int(cl(F,E))$ $%
\widetilde{\subseteq }$ $(F,E)$.Then

$int(cl(F,E))$ $\widetilde{\subseteq }$ $(F,E)$ $\widetilde{\subseteq }$ $%
cl(int(cl(F,E)))$. Since $int(cl(F,E))=(U,E)$ is a soft open set, we can
write $(U,E)$ $\widetilde{\subseteq }$ $(F,E)$ $\widetilde{\subseteq }$ $%
cl(U,E)$. Hence $(F,E)$ is a soft semi-open set.
\end{proof}

\begin{corollary}
If a soft set $(F,E)$ in a soft topological space $(X,\tau ,E)$ is soft $%
\beta $-closed and soft $semi$-open, then $(F,E)$ is soft $semi$-closed.
\end{corollary}

\begin{proposition}
In an indiscrete soft topological space $(X,\tau ,E),$ each soft $\beta $%
-open is soft $pre$-open.
\end{proposition}

\begin{proof}
If $(F,E)=\Phi $, then $(F,E)$ is a soft $\beta $-open and a soft $pre$%
-open. Let $(F,E)\neq \Phi ,$ then,

$(F,E)\in S.\beta .O(X)\Longrightarrow (F,E)$ $\widetilde{\subseteq }$ $%
cl(int(cl(F,E)))=\widetilde{X}=(int(cl(F,E))$. Hence $(F,E)$ is a soft $pre$%
-open.
\end{proof}

\begin{theorem}
A soft set $(F,E)$ in a soft topological space $(X,\tau ,E)$ is soft $\beta $%
-closed if and only if
\end{theorem}

$cl(\widetilde{X}-cl(int(F,E)))-(\widetilde{X}-cl(F,E))$ $\widetilde{%
\supseteq }$ $cl(F,E)-(F,E)$.

\begin{proof}
$cl(\widetilde{X}-cl(int(F,E)))-(\widetilde{X}-cl(F,E))$ $\widetilde{%
\supseteq }$ $cl(F,E)-(F,E)\Longleftrightarrow (\widetilde{X}%
-int(cl(int(F,E))))-(\widetilde{X}-cl(F,E))$ $\widetilde{\supseteq }$ $%
cl(F,E)-(F,E)\Longleftrightarrow (\widetilde{X}-int(cl(int(F,E))))$ $%
\widetilde{\cap }$ $cl(F,E)$ $\widetilde{\supseteq }$ $cl(F,E)-(F,E)%
\Longleftrightarrow (\widetilde{X}$ $\widetilde{\cap }$ $%
cl(F,E))-[int(cl(int(F,E)))$ $\widetilde{\cap }$ $cl(F,E)]$ $\widetilde{%
\supseteq }$ $cl(F,E)-(F,E)\Longleftrightarrow cl(F,E)-int(cl(int(F,E)))$ $%
\widetilde{\supseteq }$ $cl(F,E)-(F,E)\Longleftrightarrow (F,E)$ $\widetilde{%
\supseteq }$ $int(cl(int(F,E)))\Longleftrightarrow $ $(F,E)$ is soft $\beta $%
-closed.
\end{proof}

\begin{theorem}
Each soft $\beta $-open and soft $\alpha $-closed set is soft closed.
\end{theorem}

\begin{proof}
Let $(F,E)\in S.\beta .O(X),(F,E)$ $\widetilde{\subseteq }$ $%
cl(int(cl(F,E))),$ since $(F,E)$ is soft $\alpha $-closed $cl(int(cl(F,E)))$ 
$\widetilde{\subseteq }$ $(F,E)$, then $cl(int(cl(F,E)))$ $\widetilde{%
\subseteq }$ $(F,E)$ $\widetilde{\subseteq }$ $%
cl(int(cl(F,E))),(F,E)=cl(int(cl(F,E)))$ which is soft closed.
\end{proof}

\begin{corollary}
Each soft $\beta $-closed and soft $\alpha $-open set is soft open.
\end{corollary}

\section{Soft $\protect\beta $-Continuous Mappings}

\begin{definition}
$[10]$ Let $(F,E)$ be a soft set $X$. The soft set $(F,E)$ is called a soft
point , denoted by $(x_{e},E)$, if for the element $e\in E,$ $F(e)=\{x\}$
and $F(e^{\prime })=\phi $ for all $e^{\prime }\in E-\{e\}.$
\end{definition}

We define the notion of soft $\beta $-continuity by using soft $\beta $-open
sets.

\begin{definition}
Let $(X,\tau ,E)$ and $(Y,\tau ^{\prime },E)$ be two soft topological
spaces. A function $f:(X,\tau ,E)\longrightarrow (Y,\tau ^{\prime },E)$ is
said to be
\end{definition}

$(i)$ $\left( \left[ 8\right] \right) $soft $semi$-continuons if $%
f^{-1}((G,E))$ is soft $semi$-open in $(X,\tau ,E)$, for every soft open set 
$(G,E)$ of $(Y,\tau ^{\prime },E)$.

\ \ \ \ \ \ \ \ \ \ \ 

$(ii)$ soft $pre$-continuons if $f^{-1}((G,E))$ is soft $pre$-open in $%
(X,\tau ,E)$, for every soft open set $(G,E)$ of $(Y,\tau ^{\prime },E)$.

\ \ \ \ \ \ \ \ \ \ \ \ 

$(iii)$ soft $\alpha $-continuons if $f^{-1}((G,E))$ is soft $\alpha $-open
in $(X,\tau ,E)$, for every soft open set $(G,E)$ of $(Y,\tau ^{\prime },E).$

\ \ \ \ \ \ \ \ \ \ \ \ 

$(iv)$ soft $\beta $-continuons if $f^{-1}((G,E))$ is soft $\beta $-open in $%
(X,\tau ,E)$, for every soft open set $(G,E)$ of $(Y,\tau ^{\prime },E).$

\ \ \ \ \ \ \ \ \ \ \ \ 

$(v)$ soft $\beta $-irresolute if $f^{-1}((G,E))$ is soft $\beta $-open in $%
(X,\tau ,E),$ for every soft $\beta $-open set $(G,E)$ of $(Y,\tau ^{\prime
},E)$.

\ \ \ \ \ \ \ \ \ \ \ \ \ \ \ \ \ \ \ \ \ \ \ 

It is clear that the class of soft $\beta $-continuity contains each of
classes soft $semi$-continuons and soft $pre$-continuons, the implications
between them and other types of soft continuities are given by the following
diagram.

\ \ \ \ \ \ \ \ \ \ \ 

soft continuity $\longrightarrow $\ \ soft $\alpha $-continuity $\
\longrightarrow $ \ soft semi-continuity\ \ \ \ \ \ \ \ \ \ \ \ \ \ \ \ \ \ 

$\ \ \ \ \ \ \ \ \ \ \ \ \ \ \ \ \ \ \ \ \ \ \ \ \ \ \ \ \ \ \ \ \ \ \ \ \ \
\ \ \downarrow $ \ \ \ \ \ \ \ \ \ \ \ \ \ \ \ \ \ \ \ \ \ \ \ \ \ \ \ \ \ \
\ $\downarrow $

\bigskip\ \ \ \ \ \ \ \ \ \ \ \ \ \ \ \ \ \ \ \ \ \ \ \ soft $pre$%
-continuity\ \ $\longrightarrow $\ \ \ \ \ soft $\beta $-continuity

\ \ \ \ \ \ \ \ \ \ \ \ \ \ 

The converses of these implications do not hold, in general, as show in the
following examples.

\begin{example}
Let $X=Y=\{x_{1},x_{2},x_{3}\}$, $E=\{e_{1},e_{2}\}$ and let the soft
topology on$\ X$ be soft indiscrete and on $Y$ be soft discrete. If we get
the mapping $f:(X,\tau ,E)\longrightarrow (Y,\tau ^{\prime },E)$ defined as
\end{example}

\ \ \ \ \ 

\ \ \ \ \ \ \ \ $\ \ \ f(x_{1})=x_{2},$ \ \ \ \ $f(x_{2})=x_{1},$\ $\ \ \ \
f(x_{3})=x_{3}$

\ \ \ \ \ \ \ \ \ \ \ \ \ \ \ \ \ \ \ \ \ \ \ \ \ \ \ \ \ \ \ \ \ \ \ \ \ \
\ \ \ \ \ \ \ \ \ \ \ \ \ \ \ \ \ \ \ \ \ \ \ \ \ \ \ \ \ \ \ \ \ \ \ \ \ \
\ \ \ \ \ \ \ \ \ 

then $f$ is soft $\beta $-continuous but not soft $semi$-continuous.

\begin{example}
Let $X=Y=\{x_{1},x_{2},x_{3}\}$, $E=\{e_{1},e_{2}\}$. Then $\tau =\{\Phi ,%
\widetilde{X},(F_{1},E),(F_{2},E),(F_{3},E)\}$ is a soft topological space
over $X$ and $\tau ^{\prime }=\{\Phi ,\widetilde{Y},(G_{1},E),(G_{2},E)\}$
is a soft topological space over $Y.$ Here $(F_{1},E),(F_{2},E),(F_{3},E)$
are soft sets over $X$ and $(G_{1},E),(G_{2},E)$ are soft sets over $Y,$%
defined as follows:
\end{example}

\ \ \ \ \ \ \ \ \ \ \ \ \ \ \ \ \ \ \ \ \ \ \ \ \ \ \ \ \ \ 

$F_{1}(e_{1})=\{x_{1}\},$ $F_{1}(e_{2})=\{x_{1}\},$ $F_{2}(e_{1})=\{x_{2}\},$
$F_{2}(e_{2})=\{x_{2}\},$ $F_{3}(e_{1})=\{x_{1},x_{2}\},$ $%
F_{3}(e_{2})=\{x_{1},x_{2}\}$

\ \ \ \ \ \ \ \ 

and

\ \ \ 

$G_{1}(e_{1})=\{x_{1}\},$ $G_{1}(e_{2})=\{x_{1}\},$ $G_{2}(e_{1})=%
\{x_{1},x_{2}\},$ $G_{2}(e_{2})=\{x_{1},x_{2}\}.$

\bigskip\ \ \ \ \ \ \ \ \ \ \ \ \ \ \ \ \ \ \ \ \ \ \ \ \ \ \ \ \ \ \ \ \ \
\ \ \ \ 

If we get the mapping $f:(X,\tau ,E)\longrightarrow (Y,\tau ^{\prime },E)$
defined as

\ \ \ \ \ 

\ \ \ \ \ \ \ \ $\ \ \ f(x_{1})=x_{1},$ \ \ \ \ $f(x_{2})=x_{3},$\ $\ \ \ \
f(x_{3})=x_{2}$

\ \ \ \ \ \ \ \ \ \ \ \ \ \ \ \ \ \ \ \ \ \ \ \ \ \ \ \ \ \ \ \ \ \ \ \ \ \
\ \ \ \ \ \ \ \ \ \ \ \ \ \ \ \ \ \ \ \ \ \ \ \ \ \ \ \ \ \ \ \ \ \ \ \ \ \
\ \ \ \ \ \ \ \ \ 

then $f$ is soft $\beta $-continuous but not soft $pre$-continuous, since $%
f^{-1}(G_{2})=\{\{x_{1},x_{3}\},\{x_{1},x_{3}\}\}$ is not a soft $pre$-open
set over $X$.

\ \ \ \ \ \ \ \ \ 

We give some characterizations of soft $\beta $-continuity.

\begin{theorem}
Let $f:(X,\tau ,E)\longrightarrow (Y,\tau ^{\prime },E)$ be a soft mapping,
then the following statements are equivalent.
\end{theorem}

$(i)$ $f$ \ is soft $\beta $-continuous.

\ \ \ \ \ \ \ \ \ \ \ \ \ \ \ \ \ \ \ \ \ \ \ \ \ \ \ \ \ \ \ 

$(ii)$ For each soft point $(x_{e},E)$ over $X$ and each soft open $(G,E)$
containing $f(x_{e},E)=(f(x)_{e},E)$ over $Y,$ there exist a soft $\beta $%
-open set $(F,E)$ over $X$ containing $(x_{e},E)$ such that $f(F,E)$ $%
\widetilde{\subseteq }$ $(G,E).$

\ \ \ \ \ \ \ \ \ \ \ \ \ \ \ \ \ \ \ 

$(iii)$ The inverse image of each soft closed set in $Y$ is soft $\beta $%
-closed in $X$.

\ \ \ \ \ \ \ \ \ \ \ \ \ \ \ \ \ \ 

$(iv)$ $int(cl(int(f^{-1}(G,E))))$ $\widetilde{\subseteq }$ $f^{-1}(cl(G,E))$
for each soft set $(G,E)$ over $Y.$

\ \ \ \ \ \ \ \ \ \ \ \ \ \ \ \ \ \ \ \ \ \ \ \ 

$(v)$ $f(int(cl(int(F,E))))$ $\widetilde{\subseteq }$ $cl(f(F,E))$ for each
soft set $(F,E)$ over $X.$

\ \ \ \ \ \ \ \ \ \ \ \ \ \ \ \ \ \ \ \ \ \ \ \ \ \ \ \ \ \ \ 

\begin{proof}
$(i)\Longrightarrow (ii)$ Since $(G,E)$ $\widetilde{\subseteq }$ $Y$
containing $f(x_{e},E)=(f(x)_{e},E)$ is soft open, then $f^{-1}(G,E)\in
S.\beta .O(X).$ Soft set $(F,E)=f^{-1}(G,E)$ which containing $(x_{e},E),$
therefore $f(F,E)$ $\widetilde{\subseteq }$ $(G,E).$

\ \ \ \ \ \ \ \ \ \ \ \ 

$(i)\Longrightarrow (iii)$ Let $(G,E)\in S.C(Y),$then $(\widetilde{Y}%
-(G,E))\in S.O(Y).$Since\ $f$ is soft $\beta $-continuous, $f^{-1}(%
\widetilde{Y}-(G,E))\in S.\beta .O(X).$ Hence $[\widetilde{X}%
-f^{-1}(G,E)]\in S.\beta .O(X).$ Then $f^{-1}(G,E)\in S.\beta .C(X)$

\ \ \ \ \ \ \ \ \ \ \ \ \ \ \ \ \ \ \ \ \ \ \ \ \ \ \ 

$(iii)\Longrightarrow (iv)$ Let $(G,E)$ be a soft set over $Y,$ then $%
f^{-1}(cl(G,E))\in S.\beta .C(X).$ $f^{-1}(cl(G,E))$ $\widetilde{\supseteq }$
$int(cl(int(f^{-1}(cl(G,E)))))$ $\widetilde{\supseteq }$ $%
int(cl(int(f^{-1}(G,E))))$

\ \ \ \ \ \ \ \ \ \ \ \ \ \ \ \ \ \ \ 

$(vi)\Longrightarrow (v)$ Let $(F,E)$ be a soft set over $X$ and $%
f(F,E)=(G,E)$. Then, according to\ $(iv)$ \ \ $int(cl(int(f^{-1}(f(F,E)))))$ 
$\widetilde{\subseteq }$ $f^{-1}(cl(f(F,E))\Longrightarrow $\ \ $%
int(cl(int(F,E)))$ $\widetilde{\subseteq }$ $f^{-1}(cl(f(F,E))$\ \ $%
\Longrightarrow $\ $\ f(int(cl(int(F,E))))$ $\widetilde{\subseteq }$ $%
cl(f(F,E))$

\ \ \ \ \ \ \ \ \ \ \ \ \ \ \ \ \ \ \ \ \ \ \ \ \ \ \ \ \ \ \ \ \ \ \ \ \ \
\ \ \ \ \ \ \ \ 

$(v)\Longrightarrow (i)$ Let $(G,E)\in S.O(Y),$ $(H,E)=$\ $\widetilde{Y}%
-(G,E)$\ and $(F,E)=$ $f^{-1}(H,E),$ by\ $(v)$\ $\
f(int(cl(int(f^{-1}(H,E)))))$ $\widetilde{\subseteq }$ $cl(f(f^{-1}(H,E)))$ $%
\widetilde{\subseteq }$ \ $cl(H,E)=$\ $(H,E),$ so $int(cl(int(f^{-1}(H,E))))$
$\widetilde{\subseteq }$ \ $f^{-1}(H,E).$ Then\ \ $f^{-1}(H,E)$\ $\in
S.\beta .C(X)$, thus (by $(iii)$) $f$ \ is soft $\beta $-continuous.\ \ \ \
\ \ \ \ 
\end{proof}

\begin{remark}
The composition of two soft $\beta $-continuous mappings need not be soft $%
\beta $-continuous, in general, as show by the following example.
\end{remark}

\begin{example}
Let $X=Z=\{x_{1},x_{2},x_{3}\}$, $Y=\{x_{1},x_{2},x_{3},x_{4}\}$ and $%
E=\{e_{1},e_{2}\}.$ Then $\tau =\{\Phi ,\widetilde{X},(F,E)\}$ is a soft
topological space over $X,$ $\tau ^{\prime }=\{\Phi ,\widetilde{Y},(G,E)\}$
is a soft topological space over $Y$ and $\tau ^{\prime \prime }=\{\Phi ,%
\widetilde{Z},(H_{1},E),(H_{2},E)\}$ is a soft topological space over $Z.$
Here $(F,E)$ is a soft set over $X$, $(G,E)$ is a soft set over $Y$ and $%
(H_{1},E),(H_{2},E)$ are soft sets over $Z$ defined as follows:
\end{example}

\ \ \ \ \ \ \ \ \ \ \ \ \ \ \ \ \ \ \ \ \ \ \ \ \ \ \ \ \ \ 

$F(e_{1})=\{x_{1}\},$ $F(e_{2})=\{x_{1}\},$ $G(e_{1})=\{x_{1},x_{3}\},$ $%
G(e_{2})=\{x_{1},x_{3}\}$

\ \ \ \ \ \ \ \ \ \ \ \ \ \ \ \ \ \ \ \ \ \ \ \ \ \ \ \ \ \ \ \ \ \ \ \ \ 

$H_{1}(e_{1})=\{x_{3}\},$ $H_{1}(e_{2})=\{x_{3}\},H_{2}(e_{1})=\{x_{1},x_{2}%
\},$ $H_{2}(e_{2})=\{x_{1},x_{2}\}$

\ \ \ \ \ \ \ \ \ \ \ \ \ \ \ \ \ \ \ \ \ \ \ \ \ \ \ \ \ \ \ \ \ \ \ \ \ \ 

If we get the mapping the identity mapping $I:(X,\tau ,E)\longrightarrow
(Y,\tau ^{\prime },E)$ and $f:(Y,\tau ^{\prime },E)\longrightarrow (Z,\tau
^{\prime \prime },E)$ defined as

\ \ \ \ \ 

\ \ \ \ \ \ \ \ $\ \ \ f(x_{1})=x_{1},$ \ \ \ \ $f(x_{2})=f(x_{4})=x_{2},$\ $%
\ \ \ \ f(x_{3})=x_{3}$

\ \ \ \ \ \ \ \ \ \ \ \ \ \ \ \ \ \ \ \ \ \ \ \ \ \ \ \ \ \ \ \ \ \ \ \ \ \
\ \ \ \ \ \ \ \ \ \ \ \ \ \ \ \ \ \ \ \ \ \ \ \ \ \ \ \ \ \ \ \ \ \ \ \ \ \
\ \ \ \ \ \ \ \ \ 

It is clear that each of $I$ and $f$ is soft $\beta $-continuous but $f$ $o$ 
$I$ not soft $\beta $-continuous, since

$(f$ $o$ $I)^{-1}(H_{1},E)=\{\{x_{3}\},\{x_{3}\}\}$ is not a soft $\beta $%
-open set over $X$.

\begin{definition}
A function $f:(X,\tau ,E)\longrightarrow (Y,\tau ^{\prime },E)$ is called
soft $\beta $-homeomorphism (resp.soft $\beta r$-homeomorphism) if $f$ is a
soft $\beta $-continuous bijection (resp. sorf $\beta $-irresolute
bijection) and $f^{-1}:(Y,\tau ^{\prime },E)\longrightarrow (X,\tau ,E)$ is
a soft $\beta $-continuous (soft $\beta $-irresolute).
\end{definition}

Now we can give the following definition by taking the soft space $(X,\tau
,E)$ instead of the soft space $(Y,\tau ^{\prime },E)$.

\begin{definition}
For a soft topological space $(X,\tau ,E),$ we define the following two
collections of functions:
\end{definition}

$S\beta $-$h(X,\tau ,E)=$\{ $f\mid f:(X,\tau ,E)\longrightarrow (X,\tau ,E)$
is a soft $\beta $-continuous bijection, $\ f^{-1}:(X,\tau
,E)\longrightarrow (X,\tau ,E)$ is soft $\beta $-continuous \}

\ \ \ \ 

$S\beta r$-$h(X,\tau ,E)=$ \{ $f\mid f:(X,\tau ,E)\longrightarrow (X,\tau
,E) $ is a soft $\beta $-irresolute bijection, $f^{-1}:(X,\tau
,E)\longrightarrow (X,\tau ,E)$ is soft $\beta $-irresolute \}

\begin{theorem}
\bigskip For a soft topological space $(X,\tau ,E)$, $S$-$h(X,\tau ,E)$ $%
\widetilde{\subseteq }$ $S\beta r$-$h(X,\tau ,E)$ $\widetilde{\subseteq }$ $%
S\beta $-$h(X,\tau ,E),$ where $S$-$h(X,\tau ,E)=\{f\mid f:(X,\tau
,E)\longrightarrow (X,\tau ,E)$ is a soft-homeomorphism$\}$ .
\end{theorem}

\begin{proof}
First we show that every soft-homeomorphism $f:(X,\tau ,E)\longrightarrow
(Y,\tau ^{\prime },E)$ is a soft $\beta r$-homeomorphism. Let $(G,E)\in $ $%
S.\beta .O(Y),$ then $(G,E)$ $\widetilde{\subseteq }$ $cl(int(cl(G,E))).$
Hence, $f^{-1}((G,E))$ $\widetilde{\subseteq }$ $%
f^{-1}(cl(int(cl(G,E))))=cl(int(cl(f^{-1}(G,E))))$ and so $f^{-1}((G,E))\in
S.\beta .O(X).$ Thus, $f$ $\ $is soft $\beta $-irresolute. In a similar way,
it is show that $f^{-1}$ is soft $\beta $-irresolute. Hence, we have that $S$%
-$h(X,\tau ,E)$ $\widetilde{\subseteq }$ $S\beta r$-$h(X,\tau ,E)$.

Finally, it is obvious that $S\beta r$-$h(X,\tau ,E)$ $\widetilde{\subseteq }
$ $S\beta $-$h(X,\tau ,E),$ because every soft $\beta $-irresolute function
is soft $\beta $-continuous.
\end{proof}

\begin{theorem}
\bigskip For a soft topological space $(X,\tau ,E)$, the collection $S\beta
r $-$h(X,\tau ,E)$ forms a group under the composition of functions.
\end{theorem}

\begin{proof}
If $f:(X,\tau ,E)\longrightarrow (Y,\tau ^{\prime },E)$ and $g:(Y,\tau
^{\prime },E)\longrightarrow (Z,\tau ^{\prime \prime },E)$ are soft $\beta r$%
-homeomorphism, then their composition $gof:(X,\tau ,E)\longrightarrow
(Z,\tau ^{\prime \prime },E)$ is a soft $\beta r$-homeomorphism. It is
obvious that for a bijective soft $\beta r$-homeomorphism $f:(X,\tau
,E)\longrightarrow (Y,\tau ^{\prime },E),$ $f^{-1}:(Y,\tau ^{\prime
},E)\longrightarrow (X,\tau ,E)$ is also a soft $\beta r$-homeomorphism and
the identity $1:(X,\tau ,E)\longrightarrow (X,\tau ,E)$ is a soft $\beta r$%
-homeomorphism. A binary operation $\alpha :S\beta r$-$h(X,\tau ,E)\times
S\beta r$-$h(X,\tau ,E)\longrightarrow S\beta r$-$h(X,\tau ,E)$ is well
defined by $\alpha (a,b)=boa,$ where $a,b\in S\beta r$-$h(X,\tau ,E)$ and $%
boa$ is the composition of $a$ and $b.$ By using the above properties, the
set $S\beta r$-$h(X,\tau ,E)$ forms a group under composition of function .
\end{proof}

\begin{theorem}
The group $S$-$h(X,\tau ,E)$ of all soft homeomorphisms on $(X,\tau ,E)$ is
a subgroup of $S\beta r$-$h(X,\tau ,E).$
\end{theorem}

\begin{proof}
For any $a,b\in S$-$h(X,\tau ,E),$ we have $\alpha (a,b^{-1})=b^{-1}o$ $a\in
S$-$h(X,\tau ,E)$ and $1_{X}\in S$-$h(X,\tau ,E)\neq \varnothing .$ Thus,
using (Theorem 4.10) and (Theorem 4.11), it is obvious that the group $S$-$%
h(X,\tau ,E)$ is a subgroup of $S\beta r$-$h(X,\tau ,E).$
\end{proof}

For a soft topological space $(X,\tau ,E),$ we can construct a new group $%
S\beta r$-$h(X,\tau ,E)$ satisfying the property: if there exists a
homeomorphism $(X,\tau ,E)\cong (Y,\tau ^{\prime },E),$ then there exists a
group isomorphism $S\beta r$-$h(X,\tau ,E)\cong S\beta r$-$h(X,\tau ,E).$

\begin{corollary}
Let $f:(X,\tau ,E)\longrightarrow (Y,\tau ^{\prime },E)$ and $g:(Y,\tau
^{\prime },E)\longrightarrow (Z,\tau ^{\prime \prime },E)$ be two functions
between soft topological spaces.
\end{corollary}

$(i)$ For a soft $\beta r$-homeomorphism $f:(X,\tau ,E)\longrightarrow
(Y,\tau ^{\prime },E),$ there exists an isomorphism, say

$f_{\ast }:S\beta r$-$h(X,\tau ,E)\longrightarrow S\beta r$-$h(X,\tau ,E),$
defined $f_{\ast }(a)=f$ $o$ $a$ $o$ $f^{-1},$ for any element $a\in S\beta
r $-$h(X,\tau ,E)$.

\ \ \ \ \ \ \ \ \ \ \ \ \ \ \ \ \ \ 

$(ii)$\ For two soft $\beta r$-homeomorphism $f:(X,\tau ,E)\longrightarrow
(Y,\tau ^{\prime },E)$ and

$g:(Y,\tau ^{\prime },E)\longrightarrow (Z,\tau ^{\prime \prime },E)$, $%
(gof)_{\ast }=g_{\ast }o$ $f_{\ast }:S\beta r$-$h(X,\tau ,E)\longrightarrow
S\beta r$-$h(Z,\tau ^{\prime \prime },E)$ holds.

\ \ \ \ \ \ \ \ \ \ \ \ \ \ \ \ \ \ \ \ \ \ \ 

$(iii)$ For the identity function $1_{X}:(X,\tau ,E)\longrightarrow (X,\tau
,E),$ $(1_{X})_{\ast }=1:S\beta r$-$h(X,\tau ,E)\longrightarrow S\beta r$-$%
h(X,\tau ,E)$ holds where $1$ denotes the identity isomorphism.

\begin{proof}
Straightforward .
\end{proof}

\end{document}